\documentclass[sn-mathphys-num]{sn-jnl}

\usepackage{graphicx}%
\usepackage{multirow}%
\usepackage{amsmath,amssymb,amsfonts}%
\usepackage{amsthm}%
\usepackage{mathrsfs}%
\usepackage[title]{appendix}%
\usepackage{xcolor}%
\usepackage{textcomp}%
\usepackage{manyfoot}%
\usepackage{booktabs}%
\usepackage{algorithm}%
\usepackage{algorithmicx}%
\usepackage{algpseudocode}%
\usepackage{listings}%
\theoremstyle{thmstyleone}%
\newtheorem{theorem}{Theorem}

\theoremstyle{thmstyletwo}%
\newtheorem{lemma}{Lemma}
\theoremstyle{thmstylethree}%
\newtheorem{definition}{Definition}%

\begin{document}

\title[Article Title]{Strongly exposed points in Orlicz-Lorentz spaces equipped with the Orlicz norm}

\author[1]{\fnm{Di} \sur{Wang}}\email{wangd267@mail2.sysu.edu.cn}

\author*[1]{\fnm{Yongjin} \sur{Li}}\email{stslyj@mail.sysu.edu.cn}


\affil[1]{\orgdiv{School of Mathematics}, \orgname{Sun Yat-sen University}, \orgaddress{\city{Guangzhou}, \postcode{510275}, \country{China}}}




\abstract{
	The criterion for a point in the unit ball to be a strongly exposed point is given. The necessity and sufficiency conditions for Orlicz-Lorentz spaces to possess strongly exposed property are given. 
	Besides, some useful methods are obtained to handle issues related to decreasing rearrangement.
}




\maketitle

\section{Introduction}\label{sec1}
~~~~
Orlicz-Lorentz space $\Lambda_{\varphi, \omega}$ is the generalization of Orlicz space $L_{\varphi}$ and Lorentz space $\Lambda_{\omega}^{p}$.
Both Orlicz space and Lorentz space are the generalization of $L_{p}$ space.
In this article we characterize the strongly exposed points, and give the necessary and sufficient condition for Orlicz-Lorentz function spaces $\Lambda_{\varphi, \omega}^{o}$ equipped with the Orlicz norm generated by arbitrary Orlicz function $\varphi$ and a decreasing weight $\omega$. Besides, according to the expression of Orlicz norm, we discuss the supporting functional of $x$ in Orlicz-Lorentz space $\Lambda_{\varphi, \omega}^{o}$ in two cases: $K(x)=\emptyset$ and $K(x)\neq \emptyset$, and obtain the description of supporting functionals at $x$.

Let $X$ be a Banach space. By $S(X)$ and $B(X)$ we denoted the unit sphere and unit ball of $X$, respectively. By $X^{*}$ and $X^{\prime}$ we denote the dual and K\"{o}the dual of $X$, respectively. Let $x\in S(X)$,
by $Grad(x)=\{f\in S(X^{*}): f(x)=\|x\|\}$ we denote the set of supporting functionals at $x$.
$x\in S(X)$ is called an exposed point of $B(X)$ if there exists a functional $f\in Grad(x)$, such that for $f(y)<1=f(x)$ for all $y\in B(X)\backslash \{x\}$.
$x\in S(X)$ is called an strongly exposed point of $B(X)$ if there exists $f\in Grad(x)$ such that for arbitrary $x_{n}\in B(X)$, $f(x_{n})\rightarrow f(x)=1$ implies $\|x_{n}-x\|\rightarrow 0\:(n\rightarrow \infty)$.
Obviously, a strongly exposed point is an exposed point.

Exposed point and strongly exposed point are important concepts in the theory of Banach space.  They have numerous application, such as in separation theory and control theory.
In 1963, J. Lindenstrauss introduced the conception of strongly exposed point in \cite{Lindenstrauss1963OOWATN}, and proved that every weakly compact convex set in a separable Banach space is the closed convex hull of its strongly exposed points(also see \cite{Troyanski1971OnLu}).
%
In 1974, R. R. Phelps showed that a Banach space E has the Radon-Nikodym Property(equivalently, every bounded subset of E is dentable) if and only if every bounded closed convex subset of E is the closed convex hull of its strongly exposed points\cite{Phelps1974}.
In 1976, J. Bourgain gave a purely geometric proof that an arbitrary convex and weakly-compact subset of $X$ is the closed convex hull of its strongly exposed points, and the set of the exposed functionals is a dense $G_{\delta}$ subset of the unit sphere
$S(X^{*})$\cite{Bourgain1976stronglyEP}.
Besides, in \cite{Greim1983StronglyEP} the necessary and sufficient conditions
for vector-valued $L_{p}$-functions to be strongly exposed were given.


From the definition of exposed point and strongly exposed point, the research of strongly exposed point in Orlicz-Lorentz space requires the results of dual space.

In 2002, H. Hudzik, H, A. Kami\'{n}ska, M. Masty\l{}o represented the Orlicz-Lorentz space in the form of Caldero\'{n}-Lozanovski$\breve{\i}$  space and gave the description of the k\"{o}the dual of the Orlicz-Lorentz space $\Lambda_{\varphi, \omega}$ generated by an Orlicz function and a regular weight by using the Lozanovski$\breve{\i}$ theorem on the representation of k\"{o}the dual space for the Caldero\'{n}-Lozanovski$\breve{\i}$ space. For more details please refer to \cite{Hudzik20021645}. It is also shown in this paper that the regularity condition on $\omega$ is necessity and sufficient for the coincidence of the Banach dual space and the described dual space.
In 2014, A. Kami\'{n}ska, K. Le\'{s}nik and
Y. Raynaud gave the K\"{o}the dual of
Orlicz-Lorentz function and sequence space
generated by $N$-function and a decreasing weight.
The proof process is quite different from \cite{Hudzik20021645}.
And The dual norm is given via the modular $P=\inf\{\int \varphi_{*}(\frac{f^{*}}{|g|})|g|: g\prec
\omega\}$ where $f^{*}$ is the decreasing rearrangement of $f$, $g\prec \omega$ means that $g$ is submajorized by $\omega$. Besides, The expression of the modular $P$ is simplified by using Halperin's level and inverse level function. Besides,
In 2019, A. Kami\'{n}ska and Y. Raynaud generalized the class of Orlicz-Lorentz spaces by replacing Orlicz sapces by general symmetric K\"{o}the function spaces,
In section 8 of \cite{Kaminska2019AbstractLS} the authors give the K\"{o}the dual of Orlicz-Lorentz space $\Lambda_{\varphi, \omega}$ generated by arbitrary Orlicz function $\varphi$ and a decreasing weight is $\mathcal{M}_{\psi, \omega}^{o}$.
The dual norm is also given.
For  more details please refer to \cite{Kaminska2019AbstractLS}.

This paper is organized as follows. In section 2 we introduce the definition and some basic properties of Orlicz-Lorentz space, and agree on some basic notation. In section 3 we introduce some auxiluary results. Section 4 is the main content of this article. in section 4.1, we give the description of $v\in (\Lambda_{\varphi, \omega}^{o})^{\prime}$ whose norm is attainable at $x$ with $K(x)=\emptyset$.

\nocite{Kaminska2013NewFF}

\section{Preliminary}
In this part we will introduce some basic definition and notation.
By $\mu_{f}(t)$ we denote the distribution function of $f$ defined by
$$
\mu_{f}(t):=\mu(\{s \in R^{+} : | f(s)|  > t\}), \;  t>0.
$$
For $f\in L_{0}$, its non-increasing rearrangement $f^*$ is defined by
$$
f^*(t):=\inf\{\lambda >0 : \mu_{f}(\lambda)\leq t \}, \;  t>0.
$$
A function $\omega: (0,\infty) \rightarrow (0, \infty)$ is called a weight function if it is a non-increasing, locally integrable with respect to the Lebesgue measure $\mu$ and
$$
\int_{0}^{\infty}\omega(t)dt=\infty.
$$
Define the function $W:R^{+}\rightarrow R^{+}$ by
$$
W(t):=\int_{0}^{t}\omega(s)ds, \quad t\geq0.
$$
The Lorentz space $\Lambda_{\omega}$ defined by
\begin{equation*}
\Lambda_{\omega}=\left\{ f\in L_{0}: \|f\|_{\Lambda_{\omega}}=\int_{R^{+}}f^{*}(t)\omega(t)dt=\int_{R^{+}}f^{*}(t)dW<\infty \right\},
\end{equation*}
and the Marcinkiewicz space $M_{W}$ is defined as
\begin{equation*}
M_{W}=\left\{ f\in L^{0}: \|f\|_{M_{W}}=\sup_{t\in R^{+}}\left( \frac{ \int_{0}^{t}f^{*}(s)ds}{W(t)} \right)<\infty \right\}.
\end{equation*}
Each one is the K\"{o}the dual of the other(see \cite{Bennett1953, Reisner1982Dual}). More result about Lorentz space please refer to \cite{Carro2007Recent}.

A function $\varphi:R\rightarrow R^{+}$ is said to be an Orlicz function\cite{Chen1996} if $\varphi$ is convex, $\varphi(0)=0$ and $\varphi(u)>0$ for all $u>0$. For arbitrary Orlicz function $\varphi$ we define its complementary function $\psi$ in the sense of Young by the formula
$$
\psi(v)=\sup_{u>0}\{| uv | -\varphi(u)\}
$$
for all $v>0$. Obviously, $\psi$ is an Orlicz function as well. Let $p(u)$ and $q(v)$ be the right-derivative of $\varphi$ and $\psi$, respectively. Both $p(u)$ and $q(v)$ are non-decreasing and satisfy:
$$
\varphi(u)=\int_{0}^{u}p(t)dt, \quad \psi(v)=\int_{0}^{v}q(s)ds \quad u\geq0, v\geq0.
$$
We recall that an Orlicz function $\varphi$ satisfies the $\Delta_{2}$-condition($\varphi\in\Delta_{2}$) if there exist $C>0$ such that $\varphi(2t)\leq C\varphi(t)$ for all $t>0$. We say $\varphi\in\nabla_{2}$ if its complementary function $\psi\in\Delta_{2}$.
Besides, $\varphi$ and $\psi$ satisfy Young Inequality:
$$
uv \leq \varphi(u)+ \psi(v), \quad u\geq0, v\geq0,
$$
and the equation $uv=\varphi(u)+\psi(v)$ hold if and only if $u\in [q_{-}(v)sign\:v,\;q(v)sign\:v]$ or $v\in [p_{-}(u)sign\: u,\; p(u)sign\: u]$, where
\begin{align*}
p_{-}(0)&=0\: \text{and}\: p_{-}(t)=sup\{p(s): 0\leq s <t\}\: \text{for}\: t>0.\\
q_{-}(0)&=0\: \text{and}\: q_{-}(t)=sup\{q(s): 0\leq s <t\}\: \text{for}\: t>0.
\end{align*}
For arbitrary Orlicz function $\varphi$ and a decreasing weight function $\omega$, the modular $\rho_{\varphi, \omega}$ is defined by
$$
\rho_{\varphi, \omega}(f)=\int_{0}^{\infty}\varphi(f^*(t))\omega(t)dt.
$$
The Orlicz-Lorentz function space $\Lambda_{\varphi, \omega}$ is defined by
$$
\Lambda_{\varphi, \omega}=\{f\in L_{0}: \rho_{\varphi, \omega}(\lambda f)=\int_{0}^{\infty}\varphi(\lambda f^*(t))\omega(t)dt < \infty ~for~some~ \lambda >0\}.
$$
It has been proven in \cite{Kaminska199029} that $\Lambda_{\varphi, \omega}$ becomes a Banach space under the Luxemburg norm
$$
\|f\|=\|f\|_{\varphi, \omega}=\inf\{\lambda>0: \rho_{\varphi, \omega}(\frac{f}{\lambda})\leq 1\}.
$$
Orlicz norm is introduced in \cite{Wu1999235}.
For any $f\in \Lambda_{\varphi, \omega}$, its Orlicz norm is defined by
$$
\|f\|^{o}=\|f\|_{\varphi, \omega}^{o}=\sup_{\rho_{\psi, \omega}(g)\leq 1}\int_{0}^{\infty}f^*(t)g^*(t)\omega(t)dt.
$$
It has been proven in theorem 2.5 of \cite{Wu1999235} that for arbitrary Orlicz function $\varphi$ and any non-increasing weight function $\omega$, the equality
\begin{equation}
\|f\|_{\varphi, \omega}^{o}=\inf_{k>0}{\frac{1}{k}\left(1+\rho_{\varphi, \omega}(kf)\right)}
\end{equation}
holds for every $f\in \Lambda_{\varphi, \omega}$.
By $E_{\varphi, \omega}$ we denote the subspace of $\Lambda_{\varphi, \omega}$ such that
$$
E_{\varphi, \omega}=\{f\in L_{0}: \rho_{\varphi, \omega}(\lambda f)=\int_{0}^{\infty}\varphi(\lambda f^*(t))\omega(t)dt < \infty ~for~all~ \lambda >0\}.
$$
For arbitrary $f\in \Lambda_{\varphi, \omega}$,
$\theta(f)$ is defined by
\begin{equation}
	\theta(f)=\inf\{\lambda>0: \rho_{\varphi, \omega}(\frac{f}{\lambda})<\infty\}.
\end{equation}
If $\varphi\in\Delta_{2}$, then $E_{\varphi, \omega}=\Lambda_{\varphi, \omega}$.
From now on, let $\Lambda_{\varphi, \omega}=(\Lambda_{\varphi, \omega}, \|\cdot\|_{\varphi, \omega})$ and $\Lambda_{\varphi, \omega}^{o}=(\Lambda_{\varphi, \omega}, \|\cdot\|_{\varphi, \omega}^{o})$
$\Lambda_{\varphi, \omega}^{o}$ is also a Banach space(see \cite{Wu1999235}, theorem 2.3).
 Orlicz-Lorentz spaces are resonant space\cite{Bennett1953}.
 \nocite{Ferreyra2019}\nocite{Ghosh2021}.

A Banach space $E=(E, \|\cdot\|_{E})$ is said to be a symmetric space if for $f \in L^{0}$, $g \in E$ and $\mu_{f}=\mu_{g}$, then $f \in E$ and $\|f\|_{E}=\|g\|_{E}$. Orlicz-Lorentz spaces are also symmetric spaces.
For more theory of symmetric space, we refer the reader to \cite{Smithies1963,Cerda1998311}.\\
$\sigma:(0,r)\rightarrow (0,r)$ $(r\leq \infty)$ is called a measure preserving transformation \cite{Bennett1953,Brudnyi19882009} if for every measurable set $E\subset (0,r)$, $\sigma^{-1}(E) $ is measurable and $\mu(\sigma^{-1}(E))=\mu(E)$.
It was given in corollary 7.6 of \cite{Bennett1953} that for a resonant measure space and for a nonnegative $\mu$-measurable function $f$, $f$ is defined on $R$ and satisfies $\lim_{t\rightarrow\infty}f^{*}(t)=0$. Then there exists a measure-preserving transformation $\sigma$ from the support of $f$ onto the support of $f^{*}$ such that $f=f^{*}\cdot\sigma\;\mu-a.e.$ on the support of $f$. More results about measure preserving transformation please refer to \cite{Ryff1970449}.

We say that $f\in L_{0}$ is submajorized by $g\in L_{0}$ and we write
$$
f\prec g \;whenever \;\int_{0}^{t}f^{*}(s)ds<\int_{0}^{t}g^{*}(s)ds, \; t>0.
$$
Notice that $\|f\|_{M_{W}}\leq 1$ if and only if $f\prec \omega$.
By $f\sim g$ we main $f^{*}=g^{*}$.


For Orlicz function $\varphi$, if $\varphi(t)$ is affine on the interval $(a,b)$, and it is not affine on either $(a-\varepsilon, b)$ or $(a, b+\varepsilon)$ for any $\varepsilon>0$. The interval $(a,b)$ is called an affine interval of $\varphi$, denoted by $AI(a, b)$.
Define the following sets
\begin{align*}
	A^{\prime}&=\bigcup\{a_{i}: AI(a_{i}, b_{i})\: \text{with}\: p_{-}(a_{i})=p(a_{i})\}.\\
	B^{\prime}&=\bigcup\{b_{i}: AI(a_{i}, b_{i})\:\text{with}\: p_{-}(a_{i})=p(a_{i})\}.\\
	       A\:&=\bigcup\{a_{i}: AI(a_{i}, b_{i})\:\text{with}\: p_{-}(a_{i})<p(a_{i})\}.\\
	       B\:&=\bigcup\{b_{i}: AI(a_{i}, b_{i})\:\text{with}\: p_{-}(a_{i})<p(a_{i})\}.
\end{align*}
For an Orlicz function $\varphi$, by $S$ we denoted the strict convex point of $\varphi$. Let $\{[a_{i}, b_{i}]\}_{i\in N}$ be the set of all AIs of $\varphi$. Obviously, $S=R^{+}\backslash \cup_{i}(a_{i}, b_{i})$.
By $S^{\prime}$ we denoted $R^{+}\backslash S$.
For a decreasing weight $\omega$, let $W(a,b)=\int_{a}^{b}\omega(t)dt$.
We say a weight function $\omega$ is regular if $\omega$ satisfies $W(2t)\geq K W(t)$ for any $t\in R^{+}$ where $K>1$ is independent of $t$.
If $\omega(t)$ is a constant in an interval $A$, and for any interval $B$ such that $ A\subsetneqq B$, $\omega(t)$ is not a constant when $t\in B$. Then $A$ is called a maximal constant interval of $\omega$. $L(\omega)$ is a set composed of the maximal constant intervals.

Then we introduce the definition of level interval and level function.

\begin{definition}\cite[Definition 3.1]{Halperin195305}
$(a_{1}, b_{1})$, with $a\leq a_{1}<b_{1}\leq b$, is called a level interval (of $v$ with respect to $\omega$), abbreviation $l.i.$, if for all $a_{1}<x<b_{1}$, $\omega(a_{1}, x)>0$ and $R(a_{1}, x)\leq R(a_{1}, b_{1})$. If the $l.i.$ is not cantained en a large $l.i.$, it is called maximal level interval, abbreviation $m.l.i.$
\end{definition}

\begin{definition}\cite[Definition 3.2]{Halperin195305}, \cite[Definition 4.2]{Kaminska2014229}
The level function of $f$ with respect to $\omega$ is denoted by $f^{0}$, and $f^{0}$ is defined by
\begin{equation*}
f^{0}(t)=\left\{
\begin{aligned}
&R(a_{n},b_{n})\omega(t) \;\;for\; t\in (a_{n}, b_{n}),\\
&\;\;\;\omega(t) \;\;\;\;\; otherwise.
\end{aligned}
\right.
\end{equation*}
The inverse level function of $\omega$ with respect to $f$ is defined by
\begin{equation*}
\omega^{f}(t) =
\left\{
\begin{aligned}
&\frac{f(t)}{R(a_{n}, b_{n})}\;\;\; for \;t\in (a_{n}, b_{n}),\\
&\;\;\omega(t) \;\;\;\; otherwise.
\end{aligned}
\right.
\end{equation*}
Where $(a_{n}, b_{n})$ is an enumeration of all maximal intervals of $f$ with respect to $\omega$, and $R(a_{n}, b_{n})=\frac{F(a_{n}, b_{n})}{W(a_{n}, b_{n})}$.
\end{definition}

\section{Auxiliary results}

\begin{lemma}\cite[Theorem 8.10, Theorem 8.8, Theorem, Theorem 6.11]{Kaminska2019AbstractLS}
Let $\omega$ be a decreasing weight and $\varphi$ be an arbitrary Orlicz function. \\
(1)
The k\"{o}the dual of Orlicz-Lorentz space
$\Lambda_{\varphi, \omega}^{o}$ is expressed as
\begin{align*}
(\Lambda_{\varphi, \omega}^{o})^{\prime}&=(\Lambda_{\varphi, \omega}, \|\cdot\|_{\psi, \omega}^{o})^{\prime}=(\mathcal{M}_{\psi, \omega}, \|\cdot\|_{\mathcal{M}_{\varphi, \omega}})=\mathcal{M}_{\psi, \omega}.
\end{align*}
with equality of corresponding norms, where
$$
\mathcal{M}_{\psi, \omega}=\{f\in L_{0}: P_{\psi, \omega}(\lambda f)<\infty \: \text{for some} \; \lambda>0\},
$$
and the modular $P$ is define by
$$
P_{\varphi, \omega}(f)=\inf\{\int_{R^{+}}\varphi(\frac{f^{*}}{|g|})|g|:g\prec \omega\}.
$$
The norm $\|\cdot\|_{\mathcal{M}_{\psi,\omega}}$
is defined by
\begin{align}
\|f\|_{(\Lambda_{\varphi, \omega}^{o})^{\prime}}
&=\|f\|_{\mathcal{M}_{\psi, \omega}}=\inf\left\{\lambda>0: \: P_{\psi, \omega}\left(\frac{f}{\lambda}\right)\leq 1\right\}.
\end{align}
If in addition we assume that $W(\infty)=\infty$, then we also have that
$$
P_{\psi, \omega}(f)
=\int_{R^{+}}\psi(\frac{(f^{*})^{0}}{\omega})\omega(t)dt
=\int_{R^{+}}\psi(\frac{f^{*}}{\omega^{f^{*}}})\omega^{f^{*}}(t)dt.
$$
(2) Let $\varphi$ satisfy $\Delta_{2}$ condition and $\int_{0}^{\infty}\omega(t)dt = W(\infty)=\infty$. Then the dual spaces $(\Lambda_{\varphi, \omega})^{*}$ and $(\Lambda_{\varphi, \omega}^{o})^{*}$ are isometrically isomorphic to their corresponding K\"{o}the dual spaces. In fact for functional $\Phi\in (\Lambda_{\varphi, \omega})^{*}(resp., \Phi\in (\Lambda_{\varphi, \omega}^{o})^{*})$, there exists  $\phi\in \mathcal{M}_{\psi, \omega}^{o}(resp., \Phi\in(\Lambda_{\varphi, \omega}^{o})^{*})$ such that
$$
\Phi(f)=L_{\phi}(f)=\int_{0}^{\infty}f(t)\phi(t)dt,\;\; f\in \Lambda_{\varphi, \omega}.
$$
and $\|\Phi\|_{(\Lambda_{\varphi, \omega})^{*}}=\|\phi\|_{\mathcal{M}_{\psi, \omega}}^{o}\;(resp., \|\Phi\|_{(\Lambda_{\varphi, \omega}^{o})^{*}}=\|\phi\|_{\mathcal{M}_{\psi, \omega}})$.
\end{lemma}

\begin{lemma} \cite[Theorem 2.5]{Kaminska199029}
For arbitrary decreasing weight $\omega$ and Orlicz function $\varphi$, the following condition are equivalent:\\
(1)$\varphi \in \Delta_{2}$,\\
(2)$\rho_{\varphi, \omega}(f)=1$ if and only if $\|f\|_{\varphi, \omega}=1$,\\
(3)For arbitrary $\varepsilon>0$, there exists $\delta>0$, such that $\rho_{\varphi, \omega}(f)\geq 1-\varepsilon$ whenever $\|f\|_{\varphi, \omega}\geq 1-\delta$,\\
(4)$\rho_{\varphi, \omega}(f_{n})\rightarrow 0$ if and only if $\|f_{n}\|_{\varphi, \omega}\rightarrow 0\:(n\rightarrow\infty)$.
\end{lemma}


For $x\in\Lambda_{\varphi, \omega}^{o}$ we define $supp\; x=\{t\in R^{+}: x(t)\neq 0\}$.

\begin{lemma}\cite[Remark 1]{WangSEP2023}Let $\varphi$ be an Orlicz function and $\omega$ be a decreasing weight. \\
(1) If Orlicz function $\varphi$ is an $N$-function, then for all $x\in S(\Lambda_{\varphi, \omega}^{o})$ we have $k^{**}<\infty$, $K(x)\neq \emptyset$.\\
(2) If $\lim_{u\rightarrow\infty}\frac{\varphi(u)}{u}=\lim_{u\rightarrow\infty}p(u)=B<\infty$ and $\lim_{u\rightarrow\infty}Bu-\varphi(u)=\infty$, then for all $x\in \Lambda_{\varphi, \omega}^{o}$, $x\neq 0$, we have $k^{**}<\infty$ and $K(x)\neq \emptyset$.\\
(3) If $\lim_{u\rightarrow\infty} \frac{\varphi(u)}{u}=B<\infty$. Let $x\in\Lambda_{\varphi, \omega}^{o}$, $x\neq 0$. If $\psi(B)\int_{0}^{\mu(supp\:x)}\omega(t)dt>1$, then $k^{**}<\infty$ and $K(x)\neq \emptyset$. \\
(4) If $\lim_{u\rightarrow\infty}\frac{\varphi(u)}{u}=B<\infty$. Let $x\in \Lambda_{\varphi, \omega}^{o}$, $x\neq 0$ and $\psi(B)\int_{0}^{m(supp\:x)}\omega(t)dt\leq 1$, then $k^{**}<\infty$ and $K(x)= \emptyset$.

In case (1)(2)(3),  $\|x\|_{\varphi, \omega}^{o}=\frac{1}{k}\left( 1+\rho_{\varphi, \omega}(kx) \right)$ if and only if $k\in K(x)$. In case (4),
$\|x\|_{\varphi, \omega}^{o}=B\int_{0}^{\infty}x^{*}(t)\omega(t)dt$.
\end{lemma}

\begin{lemma}
For arbitrary $f\in \mathcal{M}_{\varphi, \omega}^{o}$,
if there exists $k>0$ such that
$$
\int_{R^{+}}\psi(p(\frac{kf^{*}(t)}{\omega^{f^{*}}}))\omega^{f^{*}}(t)=1
$$
then
$$
\|f\|_{\mathcal{M}_{\varphi, \omega}}^{o}=\frac{1}{k}\left(1+P_{\varphi, \omega}(kf)\right).
$$
\end{lemma}

For arbitrary $f\in \mathcal{M}_{\varphi, \omega}$, define
\begin{align*}
k^{*}_{\mathcal{M}} &=\inf\{k>0: \int_{0}^{\infty}\psi(p(\frac{kf^{*}(t)}{\omega^{f^{*}}(t)}))\omega^{f^{*}}(t)dt\geq 1\},\\
k^{**}_{\mathcal{M}}&=\sup\{k>0: \int_{0}^{\infty}\psi(p(\frac{kf^{*}(t)}{\omega^{f^{*}}(t)}))\omega^{f^{*}}(t)dt\leq 1\}.
\end{align*}
Obviously, $k^{*}_{\mathcal{M}}\leq k^{**}_{\mathcal{M}}$. Let $K_{\mathcal{M}}=[k_{\mathcal{M}}^{*}, k^{**}_{\mathcal{M}}]$. If $k^{**}_{\mathcal{M}}<\infty$,
Then $K_{\mathcal{M}}\neq \emptyset$. Next we will prove $K_{\mathcal{M}}\neq \emptyset$ when $\varphi$ is an $N$-funtion.

\begin{lemma}\cite{Wang2024}
If $\varphi$ is an N-function, then $k^{**}_{\mathcal{M}}<\infty$ for arbitrary $f\in \mathcal{M}_{\varphi, \omega}^{o}$.
\end{lemma}

\begin{lemma}\cite[Corollary 4.3]{Bennett1953}
The Banach space dual $X^{*}$ of a Banach function space $X$ is canonically isometrically isomorphic to the K\"{o}the dual
$X^{\prime}$ if and only if $X$ has absolutely continuous norm.
\end{lemma}

\begin{lemma}\cite[Theorem 9, p.g. 36]{kan1982},
For arbitrary $f\in (\Lambda_{\varphi, \omega}^{o})^{*}$, f has a unique decomposition
$$
f=L_{g}+s,\;\;g\in \mathcal{M}_{\psi, \omega},\: s\in F
$$
where s is singular and $F$ is the set of singular functionals of $x$.
\end{lemma}

Denote $\|f\|_{(\Lambda_{\varphi, \omega}^{o})^{*}}$ and $\|f\|_{\Lambda_{\varphi, \omega}^{*}}$ by $\|f\|$ and $\|f\|^{o}$, respectively.
\begin{lemma}\cite[Lemma 1.49]{Chen1996}
Let arbitrary singular functional $s\in F$,
$\|s\|=\|s\|^{o}=sup\{s(u):\rho_{\varphi,\omega}(u)<\infty\}=sup\{\frac{s(u)}{\theta(u)}:u\in L_{\varphi, \omega}\backslash E_{\varphi, \omega}\}$.
\end{lemma}
\begin{lemma}\cite[Corollary 1.50]{Chen1996}
For $f\in (\Lambda_{\varphi, \omega})^{*}$,
$\|f\|=\|f\|^{o}$ if and only if $f\in F$.
\end{lemma}

\begin{lemma}\cite[Theorem 2.48]{Chen1996}
Let $x\in \Lambda_{\varphi, \omega}$, and $\theta(x)\neq 0$. Then there exist two different singular functionals $s_{1}$ and $s_{2}$ such that $s_{1}(x)=s_{2}(x)=\theta(x)$.
\end{lemma}


\begin{lemma}\cite{Wang2024}\\
	For arbitrary $f\in (\Lambda_{\varphi, \omega}^{o})^{*}$, $f=L_{v}+s$, we have $\|f\|=\|v\|_{\mathcal{M}_{\psi, \omega}}+\|s\|$.\\
\end{lemma}

\begin{lemma}
	For any $f\in \Lambda_{\varphi, \omega}$, $d(f)=d^{o}(f)=\theta(f)$, where
	$$
	d(f)=\inf\{\|f-f_{e}\|: \: f_{e}\in E_{\varphi,\omega}\};\;d^{o}(f)=\inf\{\|f-f_{e}\|_{\varphi, \omega}^{o}:\: f_{e}\in E_{\varphi, \omega}\}.
	$$
\end{lemma}

\backmatter

\begin{lemma}\cite[Theroem 2.7 in Chapter 1]{Bennett1953}
Every Banach function space $X$ coincides with its second associate space $X^{\prime \prime}$. In other words, a function $f$ belongs to $X$ if and only if it belongs to $X^{\prime \prime}$, and in that case
\begin{equation*}
\|f\|_{X}=\|f\|_{X^{\prime \prime}}.
\end{equation*}
\end{lemma}

\begin{lemma}\cite[Proposition 1.7 in Chapter 1]{Bennett1953}
Suppose $f, g$ and $f_{n}$, $(n=1,2...)$, belong to $L_{0}$.
The decreasing rearrangement $f^{*}$ is a nonnegative, decreasing, right-continuous function on $[0, \infty)$. Furthermore,
\begin{equation*}
|f|\leq \liminf_{n\rightarrow\infty}|f_{n}|\;\mu-a.e.\;\; \Rightarrow\; f^{*}\leq \liminf_{n\rightarrow\infty}f_{n}^{*}.
\end{equation*}
In particular,
\begin{equation*}
|f_{n}|\uparrow |f|\;\mu-a.e. \;\;\Rightarrow \;f_{n}^{*}\uparrow f^{*}.
\end{equation*}
\end{lemma}

\begin{lemma}\cite[Theorem 2.2 in Chapter 1]{Bennett1953}
If $f$ and $g$ belong to $L_{0}$. Then
\begin{equation*}
\int_{R}|fg|d\mu\leq \int_{0}^{\infty}f^{*}(t)g^{*}(t)dt.
\end{equation*}
\end{lemma}

\begin{lemma}\cite[Theorem 14]{WangDiExposedpoint}
	Let $\varphi$ be an Orlicz function and $\omega$ be a decreasing weight.
	$f=L_{v}+s(0\neq v\in \mathcal{M}_{\psi, \omega}, s\in F)$ is norm attainable at $x\in S(\Lambda_{\varphi, \omega}^{o})$ ($K(x)\neq \emptyset$) if and only if\\
	(1)$v(t)=v^{*}(\sigma(t))sign x(t)$ \\
	(2)$s(kx)=\|s\|$.\\
	(3)$P_{\varphi, \omega}(\frac{v}{\|f\|})+\frac{\|s\|}{\|f\|}=1$.\\
	(4)$\int_{0}^{\infty}\frac{kv(t)x(t)}{\|f\|}dt=\rho_{\varphi, \omega}(kx)+P_{\psi,\omega}(\frac{v}{\|f\|})$ where
	$k\in K(x)$.\\
\end{lemma}

\begin{lemma}\cite[Theorem 3.7]{Wang2024}
	For arbitrary Orlicz function $\varphi$ and decreasing weight $\omega$. Let $x\in S(\Lambda_{\varphi, \omega}^{o})$ and $K(x)\neq \emptyset$. Then $Grad(x)\subset \mathcal{M}_{\psi, \omega}$ 
	if and only if one of the following conditions is satisfied.\\
	(1) $\theta(kx)<1$, $k\in K(x)$.\\
	(2) $P_{\psi,\omega}(p_{-}(kx^{*})\omega)=\rho_{\psi, \omega}(p_{-}(kx))=1$.
\end{lemma}

\begin{lemma}\cite[Theorem 3.8]{Wang2024}
	Assume $\varphi$ is an Orlicz function and $\omega$ is a decreasing weight. $L_{v}\in S(\mathcal{M}_{\psi, \omega})$ is a supporting functional of $x\in \Lambda_{\varphi, \omega}^{o}$($K(x)\neq \emptyset$) if and only if the following conditions are satisfied.\\
	(1) $P_{\psi, \omega}(v)=1$, $v(t)=v^{*}(\sigma(t))sign\:x(t)$. \\
	(2) $p_{-}(kx^{*}(t))\omega(t)\leq v^{*}(t)\leq p(kx^{*}(t))\omega(t)$ a.e in $R_{+}$ where $k\in K(x)$ and $\sigma$ is a measure preserving transformation such that $|x(t)|=x^{*}(\sigma(t))$, i.e.,
	\begin{equation*}
		p_{-}(k|x(t)|)\omega(\sigma(t))\leq v\:sign \:x(t)\leq p(k|x(t)|)\omega(\sigma(t)), a.e.\:on\:R_{+}.
	\end{equation*}
\end{lemma}

\section{Main Results}
\subsection{Strongly exposed point in Orlicz-Lorentz space}

\begin{theorem}\cite{WangDiExposedpoint}
Let $\varphi$ be an Orlicz function and $\omega$ be a decreasing weight, $\lim_{u\rightarrow \infty}\frac{\varphi(u)}{u}=B<\infty$. Assume $x_{0}\in S(\Lambda_{\varphi, \omega}^{o})$ and $\psi(B)\int_{0}^{\mu(supp\: x_{0})}\omega(t)dt<1$. Then $x_{0}$ is not an exposed point of $B(\Lambda_{\varphi, \omega}^{o})$. In addition $x_{0}$ is not a strongly exposed point of $B(\Lambda_{\varphi, \omega}^{o})$.
\end{theorem}

\begin{theorem}\cite{WangDiExposedpoint}
Assume $\varphi$ is an Orlicz function and $\omega$ is a decreasing weight.
$x\in S(\lambda_{\varphi, \omega}^{o})$, $x(t)=x^{*}(\sigma(t))$ and $k\in K(x)$. If $x=\alpha\chi_{A}$ where $k\alpha\in S^{\prime}$ and $\mu(\sigma(A)\cap L(\omega))=0$. Then $x$ is not an exposed point of $B(\Lambda_{\varphi, \omega})$. In addition, $x$ is not a strongly exposed point of $B(\Lambda_{\varphi, \omega}^{o})$.
\end{theorem}

\begin{theorem} Let $\varphi$ be an Orlicz function and $\omega$ be a decreasing weight.
$x\in S(\Lambda_{\varphi, \omega}^{o})$ and $K(x)\neq \emptyset$. Then $x$ is a strongly exposed point of $B(\Lambda_{\varphi, \omega}^{o})$ if and only if \\
(1) $\varphi\in\Delta_{2}$.\\
(2) $K(x)=\{ k\}$ and $\mu\{t\in R^{+}: kx(t)\notin S \}=0$; \\
(3) There exists a supporting functional $v(t)\in S(\mathcal{M}_{\psi, \omega})$ of $x$ and $\delta>0$ such that $P_{\psi, \omega}((1+\delta)v)<\infty$\\
(4)$\mu E=\{t\in R^{+}: kx^{*}(t)\in A^{\prime}\cup B^{\prime}\}=0$ \\
\end{theorem}

\begin{proof}
\textsl{Necessity. }

Since an strongly exposed point is an strongly extreme point, (1) and (2) is obviously. Assume (3) is not necessary, then for any $v(t)\in S(\mathcal{M}_{\psi, \omega})$ and $\varepsilon>0$, if $\int_{R^{+}}v(t)x(t)dt=1$, then $P_{\psi, \omega}((1+\varepsilon)v(t))=\infty$.
$\|v\chi_{R^{+}\backslash G_{n}}\|_{\mathcal{M}_{\psi,\omega}}=1$
where $G_{n}=\{t\in R^{+}:|v(t)|\leq n\}$. From the definition of $\|v\chi_{R^{+}\backslash G_{n}}\|_{\mathcal{M}_{\psi, \omega}}$ there exist $u_{n}(t)=u_{n}\chi_{R^{+}\backslash G_{n}}(t)$, such that $\int_{R^{+}}v(t)u_{n}(t)dt=\int_{R^{+}\backslash G_{n}}v(t)u_{n}(t)dt\rightarrow 1$. Let
\begin{equation*}
x_{n}(t)=\frac{1}{2}(x(t)+u_{n}(t))
\end{equation*}
then from the fact that $\int_{R^{+}}x_{n}(t)v(t)dt\rightarrow 1$ as $n\rightarrow\infty$, we have
\begin{align*}
1&\geq \frac{1}{2}\left( \|x\|_{\varphi, \omega}^{o}+\|x_{n}\|_{\varphi, \omega}^{o}  \right)
\geq \|x_{n}\|_{\varphi, \omega}^{o}\\
&\geq \frac{1}{2}\left(\int_{G_{n}}x_{n}(t)v(t)dt+\int_{R\backslash G_{n}}u_{n}(t)v(t)dt\right)\rightarrow 1(n\rightarrow \infty).
\end{align*}
Therefore $\|x_{n}\|_{\varphi, \omega}^{o}\rightarrow 1$ and $\int_{R^{+}}x_{n}(t)v(t)dt\rightarrow 1$ as $n\rightarrow \infty$.
While $\|x-x_{n}\|_{\varphi, \omega}^{o}\geq \frac{1}{2}\|x\chi_{G_{n}}\|_{\varphi, \omega}^{o}\rightarrow \frac{1}{2}$, therefore $x$ is not a strongly exposed point.

Then we will prove the necessity of (4). If condition (4) is not satisfied, then we will discuss the following two cases.\\
(1) $\mu\{t\in R^{+}: kx^{*}(t)\in A^{\prime}\}\neq 0$.\\
(2) $\mu\{t\in R^{+}: kx^{*}(t)\in B^{\prime}\}\neq 0$.\\
We will discuss (2), (1) is similar.
Define $x_{n}\in \Lambda_{\varphi, \omega}^{o}$ such that
\begin{equation*}
kx_{n}(t)=kx(t)\chi_{R^{+}\backslash \sigma^{-1}(G(b_{i}))}+ (b_{i}-\varepsilon_{0}-\frac{1}{n})\chi_{\sigma^{-1}(G(b_{i}))}.
\end{equation*}
where $G(b_{i})=\{t\in R^{+}: kx^{*}(t)=b_{i}\}$. 
Obviously, $x_{n}\rightarrow x_{0}$ as $n\rightarrow\infty$, where 
$$
kx_{0}(t)=kx(t)\chi_{R^{+}\backslash\sigma^{-1}(G(b_{i}))}+
(b_{i}-\varepsilon_{0})\chi_{\sigma^{-1}(G(b_{i}))}.
$$
Consider $x_{n}^{\prime}$ defined as
\begin{align*}
kx_{n}^{\prime}
&=kx^{*}\chi_{R^{+}\backslash G(b_{i})}
+(kx^{*}-\varepsilon_{0}-\frac{1}{n})\chi_{G(b_{i})}\\
&=kx^{*}\chi_{R^{+}\backslash G(b_{i})}+(b_{i} -\varepsilon_{0}-\frac{1}{n})\chi_{G(b_{i})}.
\end{align*}
Since
\begin{align*}
\mu_{kx_{n}}(\lambda)
&=\mu\{t\in R^{+}: |kx_{n}|>\lambda\}\\
&=\mu\{t\in R^{+}\backslash \sigma^{-1}(G(b_{i})):|kx|>\lambda\}+\mu \{t\in \sigma^{-1}(G(b_{i})):  |b_{i}-\varepsilon_{0}-\frac{1}{n}|>\lambda\}\\
&=\mu\{t\in R^{+}\backslash G(b_{i}):|kx^{*}|>\lambda\} +\mu \{t\in G(b_{i}):  |b_{i}-\varepsilon_{0}-\frac{1}{n}|>\lambda\}\\
&=\mu_{kx_{n}^{\prime}}(\lambda).
\end{align*}
Then $x_{n}$ and $x_{n}^{\prime}$ have the same distribution function. Take into consideration that $x_{n}^{\prime}$ is decreasing, we can obtain $x_{n}^{\prime}=x_{n}^{*}$. Then we have
\begin{align*}
\rho_{\psi, \omega}(p(kx_{n}))
&=\int_{R^{+}}\psi(p(k(x_{n})^{*}))\omega(t)dt\\
&\geq \int_{R^{+}}\psi(p(kx_{n}^{\prime}))\omega(t)dt\\
&\geq \int_{R^{+}\backslash G(b_{i})}\psi(p(kx^{*}))\omega(t)dt+ \int_{G(b_{i})}\psi(p(b_{i}-\varepsilon_{0}-\frac {1}{n}))\omega(t)dt\\
&\geq \int_{R^{+}\backslash G(b_{i})}\psi(p(kx^{*}))\omega(t)dt+ \int_{G(b_{i})}\psi(p(b_{i}))\omega(t)dt\\
&=\rho_{\psi, \omega}(p(kx))\\
&= 1.
\end{align*}
Since $x_{n}\leq x$, we have $(x_{n})^{*}\leq x^{*}$ and
\begin{align*}
1\leq \rho_{\psi, \omega}(p(kx_{n}))\leq\rho_{\psi,\omega}(p(kx))=1.
\end{align*}
We have $\rho_{\psi, \omega}(p(kx_{n}))=1$ and $k\|x_{n}\|_{\varphi, \omega}^{o}\in K(\frac{x}{\|x_{n}\|_{\varphi, \omega}^{o}})$.
Let $\frac{f}{\|f\|}\in Grad(x)$, $f=L_{v}+s$ where $v=p(kx^{*}(\sigma(t)))\omega(\sigma(t))$, Then $f$ is norm attainable at $x$.
From Lemma 16, and the fact that $\varphi\in\Delta_{2}$, we have $s=0$ and
\begin{align*}
\|f\|=f(x)
&=\frac{1}{k}\int_{R^{+}}kx(t)p(kx)\omega(\sigma(t))dt \\
&=\frac{1}{k}\int_{R^{+}}kx^{*}(t)p(kx^{*}(t))\omega(t)dt\\
&=\frac{1}{k}\left(P_{\psi, \omega}(p(kx^{*})\omega)+\rho_{\varphi, \omega}(kx)\right)\\
&=\frac{1}{k}\left(\rho_{\psi, \omega}(p(kx))+\rho_{\varphi, \omega}(kx)\right)\\
&=\frac{1}{k}\left(1+\rho_{\varphi, \omega}(kx)\right)\\
&=1.
\end{align*}
Then we will consider $f(\frac{x_{n}}{\|x_{n}\|_{\varphi,
\omega}^{o}})$.
First we discuss
\begin{align*}
\int_{R^{+}}x_{n}p(kx^{*}(\sigma(t)))\omega(\sigma(t))dt.
\end{align*}
Define $F(t)=kx_{n}(t)p(kx^{*}(\sigma(t)))\omega(\sigma(t))$.
Since
\begin{align*}
kx_{n}(\sigma^{-1}(t))
&=\left(kx-\left(\varepsilon_{0}+\frac{1}{n}\right) \chi_{\sigma^{-1}(G^{-}_{\varepsilon}(b_{i}))}\right)(\sigma^{-1}(t))\\
&=kx(\sigma^{-1}(t))-\left(\varepsilon_{0}+\frac{1}{n}\right)\chi_{\sigma^{-1}(G^{-}_{\varepsilon}(b_{i}))}(\sigma^{-1}(t))\\
&=kx(\sigma^{-1}(t))-\left(\varepsilon_{0}+\frac{1}{n}\right)\chi_{G^{-}_{\varepsilon}(b_{i})}\\
&=kx^{*}(t)-\left(\varepsilon_{0}+\frac{1}{n}\right)\chi_{G^{-}_{\varepsilon}(b_{i})}.
\end{align*}
Therefore we have
\begin{align*}
F(\sigma^{-1}(t))=kx_{n}(\sigma^{-1}(t))(t)p(kx^{*}(t))\omega(t)
=(kx^{*}(t)-(\varepsilon_{0}+\frac{1}{n})\chi_{G(b_{i})})p(kx^{*}(t))\omega(t).
\end{align*}
Since $\sigma$ is a measure preserving transformation, it follows that
\begin{align*}
\int_{R^{+}}F(t)dt =\int_{R^{+}}F(\sigma^{-1}(t))dt.
\end{align*}
Since $\varphi(x_{n})\chi_{\sigma^{-1}(G(b_{i}))}$ and $\varphi(x_{n}^{\prime}(\sigma(t)))\chi_{\sigma^{-1}(G(b_{i}))}$ have the same distribution function,
we have
\begin{align*}
f(\frac{x_{n}}{\|x_{n}\|_{\varphi, \omega}^{o}})
&=\frac{1}{\|x_{n}\|_{\varphi, \omega}^{o}}\left(\int_{R^{+}}x_{n}p(kx^{*}(\sigma(t)))\omega(\sigma(t))dt\right)\\
&=\frac{1}{k\|x_{n}\|_{\varphi, \omega}^{o}} \left( \int_{R^{+}}kx_{n}p(kx^{*}(\sigma(t)))\omega(\sigma(t))dt\right)\\
&=\frac{1}{k\|x_{n}\|_{\varphi, \omega}^{o}}\left( \int_{R^{+}}\left(kx^{*}-\left(\varepsilon_{0}
+\frac{1}{n}\right)\chi_{G(b_{i})}\right)p(kx^{*})\omega(t)dt\right)\\
&=\frac{1}{k\|x_{n}\|_{\varphi, \omega}^{o}}\left( \int_{R^{+}}\varphi\left(kx^{*}-\left(\varepsilon_{0}
+\frac{1}{n}\right)\chi_{G(b_{i})}\right)\omega(t)dt
+\int_{R^{+}}\psi(p(kx^{*}))\omega(t)dt\right)\\
&=\frac{1}{k\|x_{n}\|_{\varphi, \omega}^{o}}\left(\int_{R^{+}}\varphi\left(kx^{*}-\left(\varepsilon_{0}+
\frac{1}{n}\right)\chi_{G(b_{i})}\right)\omega(t)dt+1\right)\\
&=\frac{1}{k\|x_{n}\|_{\varphi, \omega}^{o}}\left(\int_{R^{+}\backslash G(b_{i})}\varphi((kx_{n})^{*})\omega(t)dt+
\int_{G(b_{i})}\varphi(x_{n}^{\prime})\omega(t)dt+1\right)\\
&=\frac{1}{k\|x_{n}\|_{\varphi, \omega}^{o}}\left(\int_{R^{+}\backslash G(b_{i})}\varphi(k(x_{n})^{*})\omega(t)dt+
\int_{G(b_{i})} \varphi((x_{n})^{*})\omega(t)dt+1 \right)\\
&=\frac{1}{k\|x_{n}\|_{\varphi, \omega}^{o}}\left(\int_{R^{+}}\varphi(k(x_{n})^{*})\omega(t)dt+1\right)\\
&=\frac{1}{k\|x_{n}\|_{\varphi, \omega}^{o}}\left(\rho_{\varphi, \omega}(kx_{n})+1\right)\\
&=1.
\end{align*}
Therefore $f(\frac{x_{n}}{\|x_{n}\|_{\varphi, \omega}^{o}})\rightarrow 1=f(x)$ and
\begin{equation*}
\liminf_{n\rightarrow \infty}\|\frac{x_{n}}{\|x_{n}\|^{o}}-x\|_{\varphi, \omega}^{o}\geq \|\frac{x_{0}}{\|x_{0}\|^{o}}-x\|_{\varphi, \omega}^{o}>0.
\end{equation*}
Therefore $x$ is not a strongly exposed point of $B(\Lambda_{\varphi, \omega}^{o})$, a contradiction.

Then we will prove the necessity of condition (5).
We will only need to consider the following two case. \\
\textsl{Case 1}: $P_{\psi, \omega}(p_{-}(kx^{*})\omega)=1$ and there exists $b_{j}\in B$ and $\varepsilon>0$ such that $p_{-}(b_{j}-\varepsilon)=p(b_{j}-\varepsilon)= p_{-}(b_{j})$, $\mu G(b_{j})= \mu\{t\in R^{+}: kx^{*}(t)=b_{j}\}>0$ and $G_{\varepsilon}^{-}(b_{j})\subset L(\omega)$.
Following the method in the proof process of the necessity of condition (4),
we can define
$$
kx_{n}= kx-(\varepsilon_{0}+\frac{1}{n})\chi_{\sigma^{-1}(G(b_{j}))}.
$$
Obviously, $x_{n}\rightarrow x_{0}$ a.e. on $R^{+}$ where $kx_{0}(t)=kx-\varepsilon_{0}\chi_{\sigma^{-1}(G(b_{j}))}$ as $n\rightarrow\infty$.
We can prove that
$\|p_{-}(kx^{*})\omega(\sigma(t))\|_{\mathcal{M}_{\psi, \omega}}=1$ and $p_{-}(kx^{*}(\sigma(t)))\omega(\sigma(t))$ is the supporting functional of $x$. It is easy to verify that
$$
\int_{R^{+}}\frac{x_{n}(t)}{\|x_{n}\|_{\varphi, \omega}^{o}}\cdot p_{-}(kx^{*}(\sigma(t)))\omega(\sigma(t))dt\rightarrow 1=\int_{R^{+}}x(t)\cdot p_{-}(kx^{*}(\sigma(t)))\omega(\sigma(t))dt
$$
while $\lim_{n\rightarrow\infty}\|\frac{x_{n}}{\|x_{n}\|^{o}}-x\|_{\varphi, \omega}^{o}\geq \|\frac{x_{0}}{\|x_{0}\|^{o}}-x\|_{\varphi,\omega}^{o}$, therefore $x$ is not a strongly exposed point of $B(\Lambda_{\varphi, \omega}^{o})$.  \\
\textsl{Case 2}. If $\theta(kx)<1$, $P_{\psi, \omega}(p(kx^{*})\omega)=1$. In this case $p(kx^{*}(\sigma(t)))\omega(\sigma(t))$ is the supporting functional of $x$. If there exists $a_{i}\in A$ such that $p(a_{i}+\varepsilon)= p(a_{i})$ and $\mu G(a_{i})=\mu \{t\in R^{+}: kx(t)=a_{i}\}>0$.
We define
$$
kx_{n}=kx +\left(\varepsilon_{0}+\frac{1}{n}\right)\chi_{\sigma^{-1}(G(a_{i}))}.
$$
$x_{n}\rightarrow x_{0}$ where $kx_{0}=kx+\varepsilon_{0}\chi_{\sigma^{-1}(G(a_{i}))}$
Since
$\rho_{\psi, \omega}(p(kx_{n}))=\rho_{\psi, \omega}(p(kx))=1$,
$\int_{R^{+}}\frac{x(t)}{\|x\|_{\varphi, \omega}^{o}}\cdot p(kx^{*}(\sigma(t)))\omega(\sigma(t))dt\rightarrow 1=\int_{R^{+}}x(t)p(kx^{*}(\sigma(t)))\omega(\sigma(t))dt$
while
$\lim_{n\rightarrow\infty}\|\frac{x_{n}}{\|x_{n}\|^{o}}-x\|\geq \|\frac{x_{0}}{\|x_{0}\|^{o}}-x\|_{\varphi, \omega}^{o}>0$,
$x$ is not a strongly exposed point of $B(\Lambda_{\varphi, \omega}^{o})$, a contradiction.
Hence we finish the proof of necessity.

\textsl{Sufficiency.}
Assume $x\in S(\Lambda_{\varphi, \omega}^{o})$, $x_{n}\in S(\Lambda_{\varphi, \omega}^{o})$, $v\in \mathcal{M}_{\psi, \omega}$ and $\int_{R^{+}}x_{n}v(t)dt\rightarrow 1$. Let $k_{n}\in K(x_{n})$, first we will prove that $\{k_{n}\}_{n\in N}$ is bounded.
Since $B(\Lambda_{\varphi, \omega}^{o})$ is $weak^{*}$ compact, then there exist $x_{n_{k}}\in S(\Lambda^{o}_{\varphi, \omega})$ and  such that $x_{n_{k}}\stackrel{weak^{*}}{\longrightarrow} x^{1}$ as $n\rightarrow \infty$. Since $\int_{R^{+}}v(t)x_{n}(t)dt\rightarrow 1= \int_{R^{+}}v(t)x(t)dt$, then $x^{1}=x$. Given $\|x\|_{\varphi, \omega}^{o}=1$, there exists $[m_{1}, m_{2}]$ such that
\begin{equation*}
\int_{m_{1}}^{m_{2}}x(t)dt=a>0.
\end{equation*}
Since $\chi_{[m_{1}, m_{2}]}$ belong to the predual of $\Lambda_{\varphi, \omega}^{o}$, we can obtain
\begin{equation*}
\lim_{n\rightarrow \infty}\int_{m_{1}}^{m_{2}}x_{n}(t)dt=\int_{m_{1}}^{m_{2}}x_{0}(t)dt=a.
\end{equation*}
Thus there exists $\delta_{1}$ and $\delta_{2}$ such that $\mu\{t\in R^{+}: |x_{n}(t)|\geq \delta_{1}\}\geq \delta_{2}$. If $\lim_{n\rightarrow \infty}k_{n}=\infty$, we can get
\begin{equation*}
1\geq \frac{1}{k_{n}}\left(1+\rho_{\varphi, \omega}(k_{n}x_{n})\right)\geq \frac{1}{k_{n}}\delta_{2}\varphi(k_{n}\delta_{1})\geq \frac{\varphi(k_{n}\delta_{1})}{k_{n}\delta_{1}}\delta_{1}\delta_{2}\rightarrow \infty, \; n\rightarrow\infty,
\end{equation*}
and it lead to a contradiction.
Therefore $\{k_{n}\}_{n\in N}$ is bounded. For arbitrary subset $A\subset R^{+}$,
since
\begin{equation}
\begin{aligned}
0
&\leftarrow \rho_{\varphi, \omega}(k_{n}x_{n})+1-\int_{R^{+}}k_{n}x_{n}(t)v(t)dt \\
&=\int_{R^{+}}\varphi(k_{n}x^{*}_{n})\omega(t)dt+\int_{R^{+}}\psi(\frac{v^{*}(t)}{\omega(t)})\omega(t)dt-\int_{R^{+}}k_{n}x_{n}(t)v(t)dt\\
&\geq\int_{R^{+}}\left(\varphi(k_{n}x_{n}(\sigma^{-1}(t)))+\psi(\frac{v^{*}(t)}{\omega(t)})\right)\omega(t)dt
-\int_{R^{+}}k_{n}x_{n}\frac{v(t)}{\omega(\sigma(t))}\omega(\sigma(t))dt\\
&= \int_{R^{+}}\left(\varphi(k_{n}x_{n}(\sigma^{-1}(t)))+\psi(\frac{v(\sigma^{-1}(t))}{\omega(t)})\right)\omega(t)dt
-\int_{R^{+}}k_{n}x_{n}(t)\frac{v(t)}{\omega(\sigma(t))}\omega(\sigma(t))dt\\
&=\int_{R^{+}}\left(\varphi(k_{n}x_{n}(t))+\psi(\frac{v(t)}{\omega(\sigma(t))})\right)\omega(\sigma(t))dt
-\int_{R^{+}}k_{n}x_{n}(t)\frac{v(t)}{\omega(\sigma(t))}\omega(\sigma(t))dt\\
&\geq\int_{A}\left(\varphi(k_{n}x_{n}(t))+\psi(\frac{v(t)}{\omega(\sigma(t))})-
k_{n}x_{n}\frac{v(t)}{\omega(\sigma(t))}\right)\omega(\sigma(t))dt.\\
&\geq 0
\end{aligned}
\end{equation}
Then
\begin{equation}
\int_{A}\left(\varphi(k_{n}x_{n}(t))+\psi(\frac{v(t)}{\omega(\sigma(t))})-k_{n}x_{n}\frac{v(t)}{\omega(\sigma(t))}\right)\omega(\sigma(t))dt\rightarrow 0\; \text{as~}n\rightarrow\infty.
\end{equation}
Besides, from formula (6) we also have
\begin{equation}
\int_{R^{+}}\varphi(k_{n}x_{n}(\sigma^{-1}(t)))\omega(t)dt-\rho_{\varphi, \omega}(k_{n}x_{n}^{*})\rightarrow 0\;\text{as~}n\rightarrow\infty.
\end{equation}
Next we will prove
\begin{align}
\lim_{\mu e\rightarrow 0}&\sup_{n}\rho_{\varphi, \omega}(k_{n}x_{n}\chi_{e})=0,\\
\lim_{\mu e\rightarrow 0}&\sup_{n}\rho_{\psi, \omega}(p(k_{n}x_{n}\chi_{e}))=0
\end{align}
where $k_{n}\in K(x_{n})$. If not, there exist $\varepsilon>0$ and $\{e_{n_{i}}\}_{i\in N}$
such that
\begin{align*}
\lim_{i\rightarrow \infty}\mu e_{n_{i}}&=0,\\
\rho_{\varphi, \omega}(k_{n_{i}}x_{n_{i}}\chi_{e_{n_{i}}})&\geq \varepsilon, \\
P_{\psi, \omega}((1+\delta)v\chi_{e_{n_{i}}})&\leq \frac{\delta\varepsilon}{2}.
\end{align*}
Thus from formula (4)(5) we have
\begin{align*}
0
&\leftarrow  1+\rho_{\varphi, \omega}(k_{n_{i}}x_{n_{i}})-k_{n_{i}}\int_{R^{+}}x_{n_{i}}v(t)dt\\
&=\int_{R^{+}}\varphi(k_{n_{i}}x^{*}_{n_{i}}(t))\omega(t)dt
+\int_{R^{+}}\psi(\frac{v^{*}(t)}{\omega})\omega(t)dt
-\int_{R^{+}}k_{n_{i}}x^{*}_{n_{i}}\frac{v^{*}(t)}{\omega(t)}\omega(t)dt\\
&= \int_{e_{n_{i}}}\left(\varphi(k_{n_{i}}x^{*}_{n_{i}}(t))
+\psi(\frac{v^{*}(t)}{\omega(t)})-k_{n_{i}}x_{n_{i}}v^{*}(t)\right)\omega(t)dt\\
&\geq \int_{e_{n_{i}}}
\left(\varphi(k_{n_{i}}x^{*}_{n_{i}}(t))+\psi(\frac{v^{*}(t)}{\omega})
-\frac{1}{1+\delta}k_{n_{i}}x^{*}_{n_{i}}(1+\delta)
\left(\frac{v^{*}(t)}{\omega(t)}\right)\right)\omega(t)dt\\
&\geq \int_{e_{n_{i}}}
\left(\varphi(k_{n_{i}}x^{*}_{n_{i}}(t))+\psi(\frac{v^{*}(t)}{\omega})
-\frac{1}{1+\delta}\left(\varphi(k_{n_{i}}x^{*}_{n_{i}})
+\psi(\frac{(1+\delta)v^{*}}{\omega})\right)\right)\omega(t)dt\\
&\geq \frac{\delta\varepsilon}{1+\delta}-\frac{P_{\psi, \omega}((1+\delta)v(t)\chi_{e_{i}})}{1+\delta}\\
&\geq \frac{\delta\varepsilon}{2(1+\delta)}.
\end{align*}
It lead to a contradiction. Therefore the formula $(9)$ is true. Since
$$\lim_{\mu e\rightarrow 0}\sup_{n}\int_{e}k_{n}x_{n}(t)p(k_{n}x_{n}(t))dt=0,$$
we obtain the formula (10) is true.

In the following, we will prove the sufficiency in the following cases.

\textsl{Case 1.}
If $P_{\psi, \omega}(p_{-}(k(x^{*}))\omega)=1$, then $\mu\{t\in R^{+}: kx^{*}(t)\in B\}=0$ 
In this case, $v=p_{-}(kx^{*}(\sigma(t)))\omega(\sigma(t))$ is the supporting functional of $x\in S(\Lambda_{\varphi, \omega}^{o})$. For any $\{x_{n}\}_{n\in N}\in S(\Lambda_{\varphi, \omega})$ with $\int_{R^{+}}x_{n}(t)v(t)dt\rightarrow 1(n\rightarrow\infty)$,
first we will prove that $k_{n}x_{n}-kx \stackrel{\mu}{\longrightarrow}0$ as $(n\rightarrow \infty)$.

Let $r_{n}$ be the discontinuous point of $p(t)$ and $a_{i}$ be arbitrary element in $A$. Define the following set
\begin{align*}
e_{i}&=\{t\in R^{+}: kx^{*}(t)=r_{i}\}.\\
E_{j}&=\{t\in R^{+}: kx^{*}(t)=a_{j}\}.\\
G_{0}&=R^{+}\backslash\cup_{i}(\sigma^{-1}e_{i})\backslash (\cup_{j}\sigma^{-1}E_{j}).
\end{align*}
Now we will prove $k_{n}x_{n}\stackrel{\mu}\rightarrow kx (n\rightarrow \infty)$ in three steps.

(1)$k_{n}x_{n}\stackrel{\mu}\rightarrow kx(n\rightarrow\infty)$ on $G_{0}$.

If not, there exist $\varepsilon>0$ and $\delta>0$ such that $\mu\{t\in R^{+}: |k_{n}x_{n}-kx|\geq \varepsilon\}\geq\delta$.
Since
\begin{align*}
k_{n}-1
\geq \rho_{\varphi, \omega}(k_{n}x_{n})
&\geq \int_{\{t\in R^{+}: |k_{n}x_{n}^{*}(t)|\geq D\}}\varphi(k_{n}x_{n}^{*}(t))\omega(t)dt\\
&\geq \varphi(D)W(\mu\{t\in R^{+}: |k_{n}x^{*}_{n}(t)|\geq D\}),\\
k-1
\geq \rho_{\varphi, \omega}(k_{n}x_{n})
&\geq \int_{\{t\in R^{+}: |kx^{*}(t)|\geq D\}}\varphi(k_{n}x_{n}^{*}(t))\omega(t)dt\\
&\geq \varphi(D)W(\mu\{t\in R^{+}: |kx^{*}(t)|\geq D\}).
\end{align*}
Following from the fact that $\{k_{n}\}_{n\in N}$ is bounded, there exists $D>0$ such that
\begin{align*}
\mu \{t\in R^{+}&: kx^{*}(t)>D\}\leq \frac{\delta}{4}.\\
\mu \{t\in R^{+}&: k_{n}x_{n}^{*}(t)>D\}\leq \frac{\delta}{4}.
\end{align*}
Since for $t\in G_{0}$, $kx^{*}(t)\neq r_{n}$, $kx^{*}(t)\neq a_{i}$, $a_{i}\in A$. Then there exist open segments $C_{n}$, $C_{i}^{\prime}$ such that $r_{n}\in C_{n}$, $a_{i}\in C_{i}^{\prime}$ and
$\mu \{ t\in R^{+}: x^{*}(t)\in (\cup_{n}C_{n})\cup(\cup_{i}C_{i}^{\prime})\}
<\frac{\delta}{4}$.
Let
$$
G_{n}=\{t\in G_{0}: |k_{n}x_{n}-kx|\geq \varepsilon,\:|k_{n}x_{n}(t)|\leq D, |kx(t)|\leq D , \:kx^{*}(t)\notin (\cup_{n}C_{n})\cup(\cup_{i}C_{i}^{\prime})\;\}
$$
Obviously, $\mu G_{n}\geq\frac{\delta}{4}$, $p_{-}(kx)\notin[p_{-}(k_{n}x_{n}), p(k_{n}x_{n})]$, and there exists an $\delta_{1}>0$ such that
\begin{align*}
\min\{|p_{-}(kx)-p_{-}(k_{n}x_{n})|, |p_{-}(kx)-p(k_{n}x_{n})|\}>\delta_{1}
\end{align*}
wherever $t\in G_{n}$. Therefore there exist $\delta_{2}>0$ such that
\begin{equation*}
\varphi(k_{n}x_{n})+\psi(kx)-k_{n}x_{n}kx\geq \delta_{2}
\end{equation*}
From formula (7) and formula (8) we can select $x_{n_{0}}$ and $M_{n_{0}}$ such that
\begin{align*}
\int_{G_{n_{0}}}\left(\varphi(k_{n_{0}}x_{n_{0}}(t))+\psi(\frac{v(t)}{\omega(\sigma(t))})-
k_{n_{0}}x_{n_{0}}\frac{v(t)}{\omega(\sigma(t))}\right)\omega(\sigma(t))dt
\leq \frac{\delta_{2}(W(M_{n_{0}}+\frac{\delta}{4})-W(M_{n_{0}}))}{2}.
\end{align*}
and
\begin{align*}
k_{n_{0}}x_{n_{0}}(\sigma^{-1}(t))&<\frac{\varepsilon}{2},\; t\geq M_{n_{0}},\\
kx^{*}(t)                         &<\frac{\varepsilon}{2},\; t\geq M_{n_{0}}.
\end{align*}
It is obviously $\sigma(G_{n_{0}})\subset (0, M_{n_{0}})$. If not, there exists $t_{0}\in G_{n_{0}}$ such that $\sigma(t_{0})>M_{n_{0}}$.
Then
\begin{align*}
k_{n_{0}}x_{n_{0}}(\sigma^{-1}(\sigma(t_{0})))
&=k_{n_{0}}x_{n_{0}}(t_{0})< \frac{\varepsilon}{2} \text{~and~} \\
kx^{*}(\sigma(t_{0}))
&=kx(t_{0})< \frac{\varepsilon}{2}.
\end{align*}
and it contribute to
\begin{align*}
|k_{n_{0}}x_{n_{0}}(t_{0})-kx(t_{0})|< \frac{\varepsilon}{2}+\frac{\varepsilon}{2}=\varepsilon, \; t_{0}\in G_{n_{0}},
\end{align*}
a contradiction, then $\sigma(G_{n_{0}})\subset (0, M_{n_{0}})$.
Thus
\begin{align*}
\frac{\delta_{2}(W(M_{n_{0}}+\frac{\delta}{4})-W(M_{n_{0}}))}{2}
&\geq
\int_{G_{n_{0}}}\left(\varphi(k_{n}x_{n}(t))+\psi(\frac{v(t)}{\omega(\sigma(t))})-
k_{n}x_{n}\frac{v(t)}{\omega(\sigma(t))}\right)\omega(\sigma(t))dt\\
&\geq \delta_{2}\int_{G_{n_{0}}}\omega(\sigma(t))dt\\
&\geq \delta_{2}(W(M_{n_{0}}+\frac{\delta}{4})-W(M_{n_{0}})),
\end{align*}
a contradiction.

(2) $k_{n}x_{n}\stackrel{\mu}{\longrightarrow}kx$ on $\sigma^{-1}(\cup_{i}e_{i})$.

If there exist $e_{i_{0}}\in \cup_{i}e_{i}$, $\varepsilon>0$ and $\delta>0$ such that
\begin{equation*}
e_{n, i_{0}}=\{t\in e_{i_{0}}:k_{n}x_{n}(t)\geq r_{i_{0}}+\varepsilon\} \text{~and~}
\liminf_{n\rightarrow \infty}
\mu e_{i_{0},n}\geq \delta.
\end{equation*}
Therefore there exists an subsequence $\{k_{n_{j}}x_{n_{j}}\}_{j\in N}\subset \{k_{n}x_{n}\}_{n\in N}$  such that
$$
p(k_{n_{j}}x_{n_{j}}(t))\geq p_{-}(k_{n_{j}}x_{n_{j}}(t))\geq p(r_{i_{0}})=p_{-}(r_{i_{0}})+\varepsilon_{0}
$$
whenever $t\in e_{n, i_{0}}$.
Then
\begin{equation*}
p_{-}(r_{i_{0}})\notin [p_{-}(k_{n_{j}}x_{n_{j}}(t)), p(k_{n_{j}}x_{n_{j}}(t))]
\end{equation*}
and
\begin{equation*}
\min\{|p_{-}(r_{i_{0}})-p_{-}(k_{n_{j}}x_{n_{j}}(t))|,\: |p_{-}(r_{i_{0}})-p(k_{n_{j}}x_{n_{j}}(t))|\}>\varepsilon_{0}.
\end{equation*}
Therefore there exists an $\varepsilon^{\prime}>0$ such that
\begin{equation*}
\varphi(k_{n_{j}}x_{n_{j}}(t))+\psi(p_{-}(r_{i_{0}}))
-k_{n_{j}}x_{n_{j}}(t)(p_{-}(r_{i_{0}}))>\varepsilon^{\prime}.
\end{equation*}
whenever $t\in e_{n, i_{0}}$.
Therefore there exist an $n_{0}$, $M_{n_{0}}$ such that

\begin{align*}
\int_{e_{n_{0}, i_{0}}}\left( \varphi(k_{n_{0}}x_{n_{0}}(t))+\psi(p_{-}(kx^{*}(t)))-k_{n_{0}}x_{n_{0}}p_{-}(kx^{*}(t))
 \right)\omega(\sigma(t))dt <\frac{\varepsilon^{\prime}}{2}\left(W(M_{n_{0}}+\delta)-W(M_{n_{0}})\right)
\end{align*}
and
\begin{align*}
k_{n_{0}}x_{n_{0}}(\sigma^{-1}(t))
< \frac{\varepsilon}{2},\;t\geq M_{n_{0}},
kx^{*}(t)
\leq \frac{\varepsilon}{2}, \;t\geq M_{n_{0}}.
\end{align*}
It is obviously
$\sigma(t)\leq M_{n_{0}}$ whenever $t\in e_{n_{0}, i_{0}}$. If not, assume there exist $t_{0}\in e_{n_{0}, i_{0}}$ such that $\sigma(t_{0})\geq M_{n_{0}}$, then $kx(t)=kx^{*}(\sigma(t)\leq \frac{\varepsilon}{2}$ and $k_{n_{0}}x_{n_{0}}(t)\leq \frac{\varepsilon}{2}$, then $|k_{n_{0}}x_{n_{0}}-kx|\leq \varepsilon$, a contradiction to $t_{0}\in e_{n_{0}, i_{0}}$.

\begin{align*}
\frac{\varepsilon^{\prime}}{2}( W(M_{n_{0}})-W(M_{n_{0}}) )
&\geq
\int_{e_{n_{0}, i_{0}}}\left( \varphi(k_{n_{0}}x_{n_{0}}(t))
+\psi(p_{-}(kx(t)))
-k_{n_{0}}x_{n_{0}}p_{-}(kx(t))
 \right)\omega(\sigma(t))dt\\
&\geq \varepsilon^{\prime}\int_{e_{n_{0}, i_{0}}}\omega(\sigma(t))dt\\
&\geq \varepsilon^{\prime}\left(W(M_{n_{0}}+\delta)-W(M_{n_{0}})\right),
\end{align*}
a contradiction.
By the same argument, we can prove
$$
\mu\{t\in e_{i}: k_{n}x_{n}\leq r_{i}-\varepsilon\}\rightarrow 0(n\rightarrow \infty),
$$
which implies
\begin{equation*}
k_{n}x_{n}(t)-kx(t)\stackrel{\mu}{\longrightarrow}0\; on \;\cup_{i} e_{i}.
\end{equation*}

(3)\; $k_{n}x_{n}(t)-kx(t)\stackrel{\mu}{\longrightarrow}0$ on $\sigma^{-1}(\cup_{i}E_{i})$.

From the result of (1) and (2), it is easy to know
$$
k_{n}x_{n}(t)-kx(t)\stackrel{\mu}{\longrightarrow} 0
\text{~on~} G\backslash \cup_{i}\sigma^{-1}(E_{i}).
$$
Define
\begin{equation}
u_{n}=\inf_{m>n}k_{m}x_{m}, \;t\in R^{+}.
\end{equation}
Then $u_{n}\leq k_{n}x_{n}$ and
\begin{align*}
u_{n}\uparrow kx \text{~on~} R^{+}\backslash \cup_{i}\sigma^{-1}(E_{i})
\end{align*}
Thus from Lemma 14 we have
\begin{align*}
u_{n}^{*}\uparrow kx^{*} \text{~on~} R^{+}\backslash \cup_{i}E_{i}.
\end{align*}
Notice that $u_{n}\leq k_{n}x_{n}$, by Fatou Lemma we have
\begin{align*}
\liminf_{n\rightarrow\infty}\int_{R^{+}}\psi(p_{-}(k_{n}x_{n}^{*})\chi_{R^{+}\backslash \cup_{i}E_{i}})\omega(t)dt
&\geq
\liminf_{n\rightarrow\infty}\int_{R^{+}}\psi(p_{-}(u_{n}^{*})\chi_{R^{+}\backslash \cup_{i}E_{i}})\omega(t)dt\\
&\geq
\int\psi(p_{-}(kx^{*}(t))\chi_{R^{+}\backslash \cup_{i}E_{i}})\omega(t).
\end{align*}
Since $\rho_{\psi, \omega}(p_{-}(k_{n}x_{n}))\leq 1=\rho_{\psi, \omega}(p_{-}(kx^{*}))$, therefore we have
\begin{equation}
\limsup_{n\rightarrow \infty}\int\psi(p_{-}(k_{n}x_{n}^{*})\chi_{\cup_{i}E_{i}})\omega(t)dt\leq \int_{R^{+}}\psi(p_{-}(kx^{*})\chi_{\cup_{i}E_{i}})\omega(t).
\end{equation}
For every $i$ and arbitrary $\varepsilon>0$, in view of $a_{i}$ is a left extreme point of a affine interval of $\varphi$, then we have
\begin{equation*}
\mu\{t\in \sigma^{-1}(E_{i}): k_{n}x_{n}(t)\leq kx(t)-\varepsilon\}\rightarrow 0 \:(n\rightarrow \infty).
\end{equation*}
If not, there exists a subsequence of $\{x_{n}\}_{n\in N}$, still denoted as $\{x_{n}\}_{i\in N}$, and $E_{i_{0}}\in \cup_{i}E_{i}$, $\varepsilon_{0}>0$, $\delta_{0}>0$ such that
\begin{equation*}
\mu\{t\in \sigma^{-1}(E_{i_{0}}): k_{n}x_{n}(t)\leq kx(t)-\varepsilon_{0}\}\geq \delta_{0}
\end{equation*}
Since there exists $D$ large enough such that
\begin{align*}
\mu\{t\in \sigma^{-1}(E_{i_{0}})&: |k_{n}x_{n}(t)|>D\}<\frac{\delta_{0}}{4},\\
\mu\{t\in \sigma^{-1}(E_{i_{0}})&: |kx(t)|>D\}< \frac{\delta_{0}}{4}.
\end{align*}
Then the bounded closed set $E^{\prime}_{i_{0}}$ defined by
\begin{equation*}
\sigma^{-1}(E_{i_{0}}^{\prime})=\{t\in \sigma^{-1}(E_{i_{0}}): |k_{n}x_{n}(t)-kx|\geq \varepsilon_{0}, |k_{n}x_{n}(t)|\leq D, |kx(t)|\leq D\}
\end{equation*}
satisfying $\mu(\sigma^{-1}(E_{i_{0}}^{\prime}))=\mu(E_{i_{0}})\geq \frac{\delta_{0}}{2}$.
Therefore therefore exists $\varepsilon^{\prime}>0$ such that
\begin{equation*}
p(kx)\geq p(k_{n}x_{n})+\varepsilon^{\prime}
\end{equation*}
when $t\in\sigma^{-1}(E_{i_{0}}^{\prime})$.
Then there exists $\varepsilon^{\prime\prime}>0$ such that
\begin{equation*}
\varphi(k_{n}x_{n}(t))+\psi(p(kx(t)))-k_{n}x_{n}(t)p(kx(t))\geq \varepsilon^{\prime\prime}
\end{equation*}
whenever $t\in\sigma^{-1}(E_{i_{0}})$.
Select $x_{n_{0}}$ and $M_{n_{0}}$ such that
\begin{equation*}
\int_{\sigma^{-1}(E_{i_{0}}^{\prime})}\left(\varphi(k_{n}x_{n}(t))+\psi(p(kx(t)))-k_{n}x_{n}p(kx)\right)\omega(\sigma(t))dt\leq \frac{\varepsilon^{\prime\prime}}{2}(W(M_{n_{0}}+\frac{\delta_{0}}{2})
-W(M_{n_{0}}))
\end{equation*}
and
\begin{equation*}
k_{n_{0}}x_{n_{0}}(\sigma^{-1}(t))\leq \frac{\varepsilon_{0}}{2}, kx^{*}(t)\leq \frac{\varepsilon_{0}}{2}\; \:\text{when}\;t\geq M_{n_{0}}.
\end{equation*}
Thus
\begin{align*}
\frac{\varepsilon^{\prime\prime}}{2}(W(M_{n_{0}}+\frac{\delta}{2})-W(M_{n_{0}}))
&\geq
\int_{\sigma^{-1}(E_{i_{0}}^{\prime})}\left(\varphi(k_{n}x_{n}(t))+\psi(p(kx(t)))-k_{n}x_{n}p(kx)\right)\omega(\sigma(t))dt\\
&\geq \varepsilon^{\prime\prime} \int_{E_{i_{0}}^{\prime}}\omega(t)dt\\
&\geq \varepsilon^{\prime\prime}(W(M_{n_{0}}+\delta_{0})-W(M_{n_{0}}))
\end{align*}
a contradiction.
Then for every $E_{i}$ and arbitrary $\varepsilon>0$, we have
\begin{equation*}
\mu\{t\in E_{i}: k_{n}x_{n}(t)\leq kx(t)-\varepsilon\}\rightarrow\infty
\; as \; n\rightarrow\infty.
\end{equation*}
Next we will prove that for arbitrary $a_{i}\in A$ and $\varepsilon>0$, we have
\begin{equation*}
\mu\{t\in \sigma^{-1}(E_{i}): k_{n}x_{n}(t)\geq kx(t)+\varepsilon\}\rightarrow\infty
\; as \; n\rightarrow\infty.
\end{equation*}
If there exist $E_{i_{1}}$, $\delta_{1}$, $\varepsilon_{1}$ such that
\begin{equation*}
\mu\{t\in\sigma^{-1}(E_{i_{1}}): k_{n}x_{n}(t)-kx(t)\geq \varepsilon_{1} \}> \delta_{1},
\end{equation*}
Then we have
\begin{equation*}
u_{n}\geq kx+\varepsilon_{1}, \;t\in \sigma^{-1}(E_{i_{1}}).
\end{equation*}
Where $u_{n}$ is defined in formula (11).
Thus there exist $u_{0}\in\Lambda_{\varphi, \omega}^{o}$ such that
\begin{align*}
u_{n}
&\uparrow u_{0}\geq kx+\varepsilon_{1}, \;t\in\sigma^{-1}(E_{i_{1}}).\\
u_{n}^{*}
&\uparrow u^{*}_{0}\geq (kx+\varepsilon_{1})^{*}= kx^{*}+\varepsilon_{1}, \;t\in E_{i_{1}}.
\end{align*}
 From Fatou Lemma we have
\begin{align*}
\limsup_{n\rightarrow \infty}\int_{R^{+}}\psi(p_{-}(k_{n}x^{*}_{n}(t)\chi_{E_{i_{1}}}))
&\geq\liminf_{n\rightarrow \infty}\int_{R^{+}}\psi(p_{-}(k_{n}x^{*}_{n}(t)\chi_{E_{i_{1}}}))\\
&\geq\liminf_{n\rightarrow\infty}\int_{R^{+}}\psi((p_{-}(u_{n}^{*})\chi_{E_{i_{1}}})\omega(t)dt\\
&\geq\int_{R^{+}}\psi(p_{-}((u_{0}^{*})\chi_{E_{i_{1}}}))\omega(t)dt.\\
&>\int_{R^{+}}\psi(p_{-}((kx^{*}(t))\chi_{E_{i_{1}}}))\omega(t)dt.
\end{align*}
Thus
\begin{equation*}
\limsup_{n\rightarrow \infty}\int_{R^{+}}\psi(p_{-}(k_{n}x^{*}_{n}(t)\chi_{\cup_{}iE_{i}}))
>\int_{R^{+}}\psi(p_{-}((kx^{*}(t))\chi_{\cup_{i}E_{i}}))\omega(t)dt.
\end{equation*}
This is a contradiction to formula (12).
Therefore $k_{n}x_{n}-kx\stackrel{\mu}\longrightarrow 0$ as $n\rightarrow \infty$ on $\cup_{i}\sigma^{-1}(E_{i})$.
Finally, we obtain $k_{n}x_{n}\stackrel{\mu}\longrightarrow kx$ on $R^{+}$ as $n\rightarrow\infty$.

\textsl{Case 2.}
$P_{\psi, \omega}(p_{-}((kx^{*}(t))\omega))<1$ and $\theta(kx)=1$.

Since $\varphi\in\Delta_{2}$, $E_{\varphi, \omega}^{o}=\Lambda^{o}_{\varphi, \omega}$. Then under the condition that $\varphi\in\Delta_{2}$, $\theta(kx)\neq1$. This case will not happen.

\textsl{Case 3.}
If $\:\theta(kx)<1$ and $P_{\psi, \omega}(p(kx^{*})\omega)=1$, then
$\mu\{t\in R^{+}: kx^{*}(t)\in A\}=0$. 

\textsl{Subcase 3.1}
$\theta(kx)<1$, $P_{\psi, \omega}(p(kx^{*})\omega)=1$ and $\mu\{t\in R^{+}: kx^{*}(t)\in A\}=0$.

In this case, $v=p(kx(t))\omega(\sigma(t))$ is the unique supporting functional of $x$.
Assume $\{x_{n}\}_{n\in N}$ such that $\int_{R^{+}}x_{n}(t)v(t)dt\rightarrow 1$.
Take $k_{n}\in K(x_{n})$.

For every $b_{i}\in B$ we define $F_{i}$ by
$$
F_{i}=\{t\in R^{+}: kx^{*}(t)=b_{i}\}.
$$
Similarly, we can proof $k_{n}x_{n}\stackrel{\mu}{\rightarrow} kx$ on $R^{+}\backslash \cup_{i} \sigma^{-1}(F_{i})$ as $n\rightarrow \infty$.
Then we will prove $k_{n}x_{n}\stackrel{\mu}{\rightarrow}kx$ as $n\rightarrow \infty$ on $\sigma^{-1}(\cup_{i}F_{i})$.
Let
$$
y_{n}=\sup_{m>n}k_{m}x_{m}, \;m,n\in N.
$$
then $y_{n}$ is decreasing and
\begin{align*}
 y_{n}\downarrow x   \qquad \: &\text{~on~} R\backslash\cup_{i}F_{i},\\
 y_{n}\downarrow x^{\prime}>x &\text{~on~} \cup_{i}F_{i}.
\end{align*}
From Lemma 14 
\begin{align*}
y_{n}^{*}&\downarrow x^{*} \qquad \quad\;\;\; \text{~on~} R\backslash\cup_{i}F_{i},\\
y_{n}^{*}&\downarrow (x^{\prime})^{*}> x^{*} ~ \text{~on~} \cup_{i}F_{i}.
\end{align*}

By Fatou Lemma, and the fact that $p(u)$ is not decreasing and continuous on the right hand
we have
\begin{align*}
\limsup_{n\rightarrow\infty}\int_{R^{+}}\psi(p((k_{n}x_{n}^{*}(t))\chi_{R^{+}\backslash \cup_{j}F_{j}})\omega(t)dt
&\leq  \limsup_{n\rightarrow \infty} \int \psi(p(y_{n}\chi_{\cup_{j}F_{j}}))\omega(t)dt\\
&\leq \int_{R^{+}}\psi(p(kx^{*})\chi_{G\backslash \cup_{j}F_{j}})\omega(t)dt.
\end{align*}
Since $\rho_{\psi, \omega}(p(k_{n}x_{n}))= 1= \rho_{\psi, \omega}(p(kx))$, we have
\begin{equation}
\liminf_{n\rightarrow\infty}\int_{R^{+}}\psi(k_{n}x_{n}^{*}(t)\chi_{\cup_{j}F_{j}})\omega(t)dt
\geq \int_{R^{+}}\psi(p(kx^{*})\chi_{\cup_{j}F_{j}})\omega(t)dt
\end{equation}
In view of $b_{j}$ is a right extreme point of an affine interval, following the same method of case 1, we have
\begin{equation*}
\mu\{t\in \sigma^{-1}(E_{i}): k_{n}x_{n}(t)\geq kx(t)+\varepsilon\}\rightarrow 0 (n\rightarrow\infty).
\end{equation*}
If there exist $i_{0}$, $\varepsilon>0$ and $\delta_{0}>0$ such that
\begin{equation*}
\mu\{t\in \sigma^{-1}(E_{i_{0}}): k_{n}x_{n}(t)\leq kx(t)-\varepsilon_{0}\}>\delta_{0}.
\end{equation*}
Then
\begin{align*}
\liminf_{n\rightarrow\infty}\int_{R^{+}}\psi(p(k_{n}x_{n}^{*})\chi_{F_{i_{0}}})\omega(t)dt
&\leq\limsup_{n\rightarrow\infty}\int_{R^{+}}\psi(p(y_{n}^{*})\chi_{F_{i_{0}}})\omega(t)dt\\
&< \int_{R^{+}}\psi(p(kx^{*})\chi_{F_{i_{0}}})\omega(t)dt.
\end{align*}
A contradiction to formula (10).

\textsl{Case 4.}
$P_{\psi, \omega}(p(kx^{*}(t))\omega(t))<1<P_{\psi, \omega}(p(kx^{*}(t))\omega(t))$.
Select $\varepsilon_{i}$ such that
\begin{align*}
p_{-}(kx^{*}(t))+\varepsilon_{i}\leq p(kx^{*}(t)), \;\mu-a.e. on\; e_{i}
\end{align*}
and
\begin{equation*}
\int_{R^{+}\backslash \cup_{i}e_{i}}\psi(kx^{*}(t))\omega(t)dt+\int_{\cup_{i}e_{i}}\psi(kx^{*}(t)+\varepsilon_{i})\omega(t)dt=1.
\end{equation*}
Let
\begin{equation*}
v(t)=p_{-}(kx^{*}(\sigma(t)))\omega(\sigma(t))\chi_{R^{+}\backslash \sigma^{-1}(\cup_{i}e_{i})}+(p_{-}(kx^{*}(t)+\varepsilon_{i}))\omega(\sigma(t))\chi_{\sigma^{-1}(\cup_{i}e_{i})}.
\end{equation*}
Obviously, $v$ is a supporting functional of $x$.
Assume there exists a sequence $\{x_{n}\}_{n\in N}$ and $\lim_{n\rightarrow \infty}\int_{R^{+}}x(t)v(t)dt=1$. By the same method in case 1 we have
\begin{equation*}
k_{n}x_{n}(t)\stackrel{\mu}{\rightarrow} kx(t)\; on \; R^{+}\backslash \cup_{i}\sigma^{-1}(e_{i}).
\end{equation*}
Then we will prove that $k_{n}x_{n}(t)\stackrel{\mu}{\rightarrow}kx(t)$ on $\sigma^{-1}(\cup_{i}e_{i})$.
If not, assume there exists an subsequence of
$\{x_{n}\}_{n\in N}$, still named as $\{x_{n}\}_{n\in N}$, and $e_{i_{0}}$, $\varepsilon_{0}>0$, $\delta_{0}>0$ such that
\begin{equation*}
\mu\{t\in \sigma^{-1}(e_{i_{0}}): k_{n}x_{n}(t)\geq r_{i_{0}}+\varepsilon_{0}\}>\delta_{0}
\end{equation*}
Analogously, there exists $D$ large enough such that
\begin{equation*}
\sigma^{-1}(e_{i_{0}}^{\prime})=\{t\in e_{i_{0}}: |k_{n}x_{n}|<D, |kx(t)|<D, |k_{n}x_{n}-kx|>\varepsilon_{0}\}\text{~and~}
\mu (\sigma^{-1}(e_{i_{0}}^{\prime}))
>\frac{\delta_{0}}{2}.
\end{equation*}
And
\begin{equation*}
p(k_{n}x_{n})\geq p_{-}(k_{n}x_{n})>p(r_{i_{0}})=p_{-}(r_{i_{0}})+\varepsilon_{i_{0}}+a_{i_{0}}.
\end{equation*}
Thus
\begin{equation*}
\varphi(k_{n}x_{n})+\psi(p(r_{i_{0}})+\varepsilon_{i_{0}})-\int_{e_{i_{0}}}k_{n}x_{n}\left(p(r_{i_{0}})+\varepsilon_{i_{0}}\right)>\varepsilon^{\prime\prime}
\end{equation*}
There exist $n_{0}$ and $M_{n_{0}}$ such that
\begin{align*}
\int_{\sigma^{-1}(e_{n_{0}}^{\prime})}
&\left(\varphi(k_{n}x_{n})+\psi(p_{-}(r_{i_{0}}+\varepsilon_{i_{0}}))
-k_{n_{0}}x_{n_{0}}(t)(p_{-}(r_{i_{0}}+\varepsilon_{i_{0}}))\right)\omega(\sigma(t))dt\\
&\leq \varepsilon^{\prime\prime}\left(W(M_{n_{0}}+\frac{\delta_{0}}{2})-W(M_{n_{0}})\right)
\end{align*}
and
\begin{align*}
k_{n}x_{n}(\sigma^{-1}(t))
&\leq\frac{\varepsilon}{2}, \;t\geq M_{n_{0}}.\\
kx^{*}(t)
&\leq \frac{\varepsilon}{2}, \;t\geq M_{n_{0}}.
\end{align*}
Thus
\begin{align*}
\frac{\varepsilon^{\prime\prime}}{2}\left(W(M_{n_{0}}+\frac{\delta_{0}}{2})-W(M_{n_{0}})\right)
&\geq \int_{\sigma^{-1}(e^{\prime}_{i_{0}})}\left(\varphi(k_{n}x_{n})+\psi(kx(t))-k_{n}x_{n}p_{-}(r_{i_{0}}+\varepsilon_{n})\right)\omega(\sigma(t))dt\\
&\geq \varepsilon^{\prime\prime}\int_{e_{i_{0}}^{\prime}}\omega(t)dt\\
&\geq \varepsilon^{\prime\prime}\left(W(M_{n_{0}}+\frac{\delta_{0}}{2})-W(M_{n_{0}})\right).
\end{align*}
a contradiction.

\end{proof}
\subsection{Strongly exposed property in Orlicz-Lorentz space}
\begin{theorem}
$\Lambda_{\varphi, \omega}^{o}$ has strongly exposed property if and only if\\
(1)$\varphi\in\Delta_{2}\cap\nabla_{2}$.\\
(2)$\varphi(u)$ is strictly convex. \\
\end{theorem}
\begin{proof}
We only need to prove the necessity of $\varphi\in \nabla_{2}$.
If $\varphi\notin\nabla_{2}$,
there exists a sequence $\{u_{n}\}_{n\in N}$
such that $u_{k}\uparrow\infty$ and $\psi((1+\frac{1}{k})u_{k})\geq 2^{k}\psi(u_{k})$. Let $b=\sum_{k=1}^{\infty}\frac{1}{2^{k}\psi(u_{k})}\leq \frac{1}{\psi(u_{1})}$. Select $t_{0}>0$ such that $b=\int_{0}^{t_{0}}\omega(t)dt$. Then we can find a sequence $\{t_{k}\}$, $t_{k}$ decrease to zero and
$$
\int_{t_{k}}^{t_{k-1}}\omega(t)dt=\frac{1}{2^{k}\psi(u_{k})}.
$$
Let
$$
f=\sum_{i=1}^{\infty}u_{k}\omega(t)\chi_{[t_{k}, t_{k-1}]}.
$$
Then every $[t_{k}, t_{k-1}]$ is a maximal level interval and $R[t_{k}, t_{k-1}]=u_{k}$. Therefore
\begin{align*}
P_{\psi, \omega}(f)&=\sum_{k=1}^{\infty}\psi(u_{k})\int_{[t_{k}, t_{k-1}]}\omega(t)dt=1,\\
P_{\psi, \omega}((1+\varepsilon)f)&>\sum_{k=1}^{\infty}2^{k}\psi(u_{k})\int_{[t_{k}, t_{k-1}]}\omega(t)dt=\infty.
\end{align*}
Then we have $\|f\|_{\mathcal{M}_{\psi, \omega}}=1$. Let $\{v_{n}\}$ satisfying
\begin{equation*}
u_{n}\in [p_{-}(v_{n}), p(v_{n})]
\end{equation*}
and define
$$
x(t)=\sum_{k=1}^{\infty}v_{k}\chi_{[t_{k}, t_{k-1}]}.
$$
Then $y\in Grad(\frac{x}{\|x\|_{\varphi, \omega}^{o}})$.
Let $A_{n}=\{t\in R^{+}: n<|y(t)|\leq 2n\}$.
 It is easy to prove $\|y\chi_{A_{n}}\|_{\mathcal{M}_{\psi, \omega}}=1$. Since $\|y\chi_{A_{n}}\|_{\mathcal{M}_{\psi,
\omega}}=\sup\left\{\int_{A_{n}} z(t)y(t)dt : \|z\|_{\mathcal{M}_{\psi, \omega}}=1\right\}$, there exist sequence $\{z_{n}\}_{n\in N}$, $z_{n}=z_{n}\chi_{A_{n}}$ and $\int_{A_{n}}z_{n}(t)y(t)dt\rightarrow 1$.
Since
$$
\|\frac{x}{\|x\|}-z_{n}\|_{\varphi, \omega}^{o}\geq \frac{\|x_{\{t\in R^{+}: |x(t)|<n\}}\|_{\varphi, \omega}^{o}}{\|x\|_{\varphi, \omega}^{o}}>0,
$$
then $x$ is not an exposed point, and $\Lambda_{\varphi, \omega}^{o}$ does not have strongly exposed property.
\end{proof}
\nocite{Foralewski2013605}
\nocite{GRZAExposed992}
\nocite{Allen1978DualsOL}
\nocite{Kaminska2009lattice}

\nocite{Halperin195305}\nocite{Ryff1970449}\nocite{Wu1999235}
\nocite{WangSEP2023}
\nocite{Kaminska2013NewFF}


\section{Acknowledgement}
The authors gratelfully acknowledge financial support from China Scholarship Council and the National Natural Science Foundation of P. R. China (Nos. 11971493).

\bibliography{sn-bibliography}

\bigskip
\begin{appendices}
\end{appendices}
\end{document}